\documentclass[12pt,twoside]{amsart}
\usepackage{amsmath}
\usepackage{amsthm}
\usepackage{amsfonts}
\usepackage{amssymb}
\usepackage{latexsym}
\usepackage[all]{xy}
\usepackage{epsfig}
\usepackage{amscd,mathrsfs,graphicx,mathrsfs}
\usepackage[numbers,sort&compress]{natbib}

\date{}
\allowdisplaybreaks[4] \footskip=15pt

\renewcommand{\uppercasenonmath}[1]{}

\numberwithin{equation}{section} \theoremstyle{plain}
\newtheorem*{theorem*}{Theorem A}
\newtheorem*{theorem**}{Theorem B}
\newtheorem{theorem}{Theorem}[section]
\newtheorem{corollary}[theorem]{Corollary}
\newtheorem*{corollary*}{Corollary}
\newtheorem{lemma}[theorem]{Lemma}
\newtheorem*{lemma*}{Lemma}
\newtheorem{proposition}[theorem]{Proposition}
\newtheorem*{proposition*}{Proposition}
\newtheorem{remark}[theorem]{Remark}
\newtheorem*{remark*}{Remark}

\newtheorem*{example*}{Example}
\newtheorem*{observation*}{Observation}
\newtheorem{definition}[theorem]{Definition}
\newtheorem*{definition*}{Definition}

\newtheorem*{conjecture*}{Conjecture}
\newtheorem*{ack*}{ACKNOWLEDGEMENTS}
\newtheorem*{fund*}{Funding}

\newcommand{\pf}{\noindent\begin {proof}}
\newcommand{\epf}{\end{proof}}

\begin{document}
\begin{center}
{\Large \bf On Gorenstein homological dimension of groups

 \footnotetext{
$\ast$ Corresponding author; e-mail address: wren@cqnu.edu.cn.}}

\vspace{0.5cm}  Yuxiang Luo, Wei Ren$^{\ast}$  \\
{\small School of Mathematical Sciences, Chongqing Normal University, Chongqing 401331, P.R. China}
\end{center}


\bigskip

\centerline { \bf  Abstract}
Let $G$ be a group and $R$ be a ring. We define the Gorenstein homological dimension of $G$ over $R$, denoted by ${\rm Ghd}_{R}G$, as the Gorenstein flat dimension of trivial $RG$-module $R$. It is proved that ${\rm Ghd}_SG \leq {\rm Ghd}_RG$ for any flat extension of commutative rings $R\rightarrow S$; in particular, ${\rm Ghd}_{R}G$ is a refinement of  ${\rm Ghd}_{\mathbb{Z}}G$ if $R$ is $\mathbb{Z}$-torsion-free. We show a Gorenstein homological version of Serre's theorem, i.e. ${\rm Ghd}_{R}G = {\rm Ghd}_{R}H$ for any subgroup $H$ of $G$ with finite index. As an application, $G$ is a finite group if and only if ${\rm Ghd}_{R}G = 0$; this is different from the fact that the homological dimension of any non-trivial finite group is infinity.

\leftskip10truemm \rightskip10truemm

\bigskip

{\noindent \it Key Words:} Gorenstein homological dimension, group ring, Gorenstein flat.\\
{\it 2010 MSC:}  20J05, 18G20, 16S34.\\

\leftskip0truemm \rightskip0truemm \vbox to 0.2cm{}

\section { \bf Introduction}

In group theory, it is a long history issue to study groups through their (co)homological properties, which arose from both topological and algebraic sources. For any group $G$, the \emph{cohomolgical dimension} ${\rm cd}_{\mathbb{Z}}G$ and the \emph{homolgical dimension} ${\rm hd}_{\mathbb{Z}}G$, is defined as the projective dimension and the flat dimension of the trivial $\mathbb{Z}G$-module $\mathbb{Z}$, respectively.

Enochs, Jenda and Torrecillas introduced the concepts of Gorenstein projective, Gorenstein injective and Gorenstein flat modules, and developed Gorenstein homological algebra \cite{EJ00}, which has its origin dated back to the study of $G$-dimension by Auslander and Bridger \cite{AB69} in 1960s. As counterparts in Gorenstein homological algebra, the Gorenstein (co)homological dimensions of groups are extensively studied, see for example \cite{ABS09, ABHS11, BDT09, Bis21, ET18, J-G14, Tal14}.

In \cite[Definition 4.5]{ABHS11}, the \emph{Gorenstein homological dimension} ${\rm Ghd}_{\mathbb{Z}}G$ of any group $G$ is defined to be the Gorenstein flat dimension of the trivial $\mathbb{Z}G$-module $\mathbb{Z}$. Analogously, we may define Gorenstein homological dimension of group $G$ over any coefficient ring $R$, denoted by ${\rm Ghd}_RG$, to be the Gorenstein flat dimension of the trivial $RG$-module $R$. The Gorenstein homological dimension of groups generalizes the notion of homological dimension of groups, in the sense that ${\rm Ghd}_RG = {\rm hd}_RG$ if ${\rm hd}_RG$ is finite.

First, we intend to compare ${\rm Ghd}_RG$ with ${\rm Ghd}_{\mathbb{Z}}G$. More general, we concern the Gorenstein homological dimensions under the extension of coefficient rings. If $R\rightarrow S$ is a flat extension of commutative rings,  we show that for any group $G$, ${\rm Ghd}_{S}G\leq {\rm Ghd}_{R}G$ holds; see Theorem \ref{thm:ringext}. This implies that if $R$ is $\mathbb{Z}$-torsion-free, then ${\rm Ghd}_{R}G$ is a refinement of ${\rm Ghd}_{\mathbb{Z}}G$; especially, ${\rm Ghd}_{\mathbb{Q}}G\leq {\rm Ghd}_{\mathbb{Z}}G$ (see Corollary \ref{cor:GhdZG}). A specific case leading to the equality $\mathrm{Ghd}_{R}G = \mathrm{Ghd}_{S}G$ is given in Proposition \ref{prop:equ}, where $G$ is a countable group and $S = R[x]/(x^{n})$.

It is natural to consider the behavior of Gorenstein homological dimensions under extensions of groups. If a subgroup $H$ of $G$ is assumed to be of finite index, it is proved in \cite[Proposition 4.11]{ABHS11} that ${\rm Ghd}_{\mathbb{Z}}H\leq {\rm Ghd}_{\mathbb{Z}}G$. Moreover, the equality ${\rm Ghd}_{\mathbb{Z}}H = {\rm Ghd}_{\mathbb{Z}}G$ holds provided that the supremum of injective length (dimension) of flat $\mathbb{Z}G$-modules, denoted by ${\rm silf}(\mathbb{Z}G)$, is finite; see \cite[Theorem 4.18]{ABHS11}.

Our result strengthens and extends \cite[Proposition 4.11]{ABHS11} and \cite[Theorem 4.18]{ABHS11}. We show in Theorem \ref{thm:groupext} that if $H$ is a subgroup of $G$ with finite index, then ${\rm Ghd}_{R}H = {\rm Ghd}_{R}G$ for any commutative ring $R$, where the assumption for the finiteness of ``${\rm silf}$'' in \cite{ABHS11} is removed. Recall that Serre's Theorem establishes an equality between cohomology dimensions of a torsion-free group and its subgroup of finite index; see details in \cite[Theorem VIII 3.1]{Bro82}. In this sense, the result can also be regarded as a Gorenstein homological version of Serre's Theorem.

There is a homological characterization for finite groups immediately; see Corollary \ref{cor:fgroup}. That is, $G$ is a finite group if and only if ${\rm Ghd}_RG = 0$ for any commutative ring $R$; this generalizes \cite[Proposition 4.12]{ABHS11}. It is worth to compare this with a well-known fact: for any non-trivial finite group $G$ one always has ${\rm hd}_\mathbb{Z}G = \infty$, since the finiteness of ${\rm hd}_\mathbb{Z}G$ implies the group $G$ is necessarily torsion-free, while every finite group is torsion. We may also compare this result with \cite[Corollary 2.3]{ET18}, which concerns Gorenstein cohomological dimension of finite groups.

\section {\bf Gorenstein homological dimension of groups}

Let $A$ be a ring, $M$ be a (left) $A$-module. An acyclic complex
$$\cdots\longrightarrow F_{1}\longrightarrow F_{0}\longrightarrow F_{-1}\longrightarrow\cdots$$
is called a \emph{totally acyclic complex of flat modules}, provided that each $F_i$ is a flat (left) $A$-module, and for any injective (right) $A$-module $I$, the complex remains acyclic after applying $I\otimes_{A}-$. The module $M$ is said to be \emph{Gorenstein flat} if there exists a totally acyclic complex of flat modules such that $M \cong \mathrm{Ker}(F_0\rightarrow F_{-1})$; see \cite{EJT93, EJ00}.

It is proved in \cite[Theorem 3.7]{Hol04} that for a right coherent ring $A$, the class of Gorenstein flat left $A$-modules is closed under extensions and direct summands. This result is extended and generalized to any ring recently by the work \cite{SS20}. The following is immediate from \cite[Corollary 4.12]{SS20}.

\begin{lemma}\label{lem:gf-closed}
Let $A$ be a ring. Then the class of Gorenstein flat $A$-modules is closed under extensions and direct summands.
\end{lemma}

The \emph{Gorenstein flat dimension} of modules is defined in the standard way by using resolutions. Let $M$ be any left $A$-module. The Gorenstein flat dimension of $M$, denoted by $\mathrm{Gfd}_{A}M$, is defined by declaring that ${\rm Gfd}_{A}M \leq n$ if and only if $M$ has a Gorenstein flat resolution $0\rightarrow G_{n}\rightarrow\cdots\rightarrow G_{1}\rightarrow G_{0}\rightarrow M\rightarrow 0$ of length $n$; see for example \cite{EJ00, Hol04}.

\begin{remark}\label{rem:GF-closed}
It is not easy to show that the class of Gorenstein flat modules is closed under extensions; however, this basic property is crucial for studying homological properties of Gorenstein flat modules. For this reason, the rings are assumed to be coherent in \cite{Hol04}; as a generalization, the notion of $GF$-closed rings is introduced in \cite{Ben09}. Thanks to \cite[Corollary 4.12]{SS20}, now we can remove the assumptions of coherent rings and $GF$-closed rings in many situations, see for example \cite[Theorem 3.14]{Hol04} and the main theorem of \cite{YL12}, when dealing with Gorenstein flat modules and Gorenstein flat dimension of modules.
\end{remark}

Recall that for any group $G$, the {\em Gorenstein homological dimension} of $G$ is defined to be the Gorenstein flat dimension of the $\mathbb{Z}G$-module $\mathbb{Z}$ with the trivial group action; see \cite[Definition 4.5]{ABHS11}. Analogously, we have the following.

\begin{definition}\label{def:Ghd}
Let $R$ be a ring. For any group $G$, the Gorenstein homological dimension of $G$ over $R$, denoted by ${\rm Ghd}_{R}G$, is defined to be the Gorenstein flat dimension of the trivial $RG$-module $R$.
\end{definition}

Let $R$ be a, $G$ be any group. Recall that the homological dimension of a group $G$ over $R$, denoted by ${\rm hd}_{R}G$, is defined to be the flat dimension ${\rm fd}_{RG}R$ of the trivial $RG$-module $R$. Since flat modules are necessarily Gorenstein flat, and the flat dimension of any Gorenstein flat module is either zero or infinity, it follows that ${\rm Ghd}_{R}G\leq {\rm hd}_{R}G$ with the equality if ${\rm hd}_{R}G$ is finite.

First, we intend to compare ${\rm Ghd}_{R}G$ with ${\rm Ghd}_{\mathbb{Z}}G$. More general, we consider Gorenstein homological dimensions of the group under extensions of coefficient rings; see Theorem \ref{thm:ringext}.

\begin{lemma}\label{lem:flatext}
Let $R\rightarrow S$ be a flat extension of commutative rings, i.e. $S$ is flat as an $R$-module. Then, for any group $G$, $RG\rightarrow SG$ is also a flat extension.
\end{lemma}

\begin{proof}
Since $S$ is a flat $R$-module, by Lazard's theorem we have $S \cong \underrightarrow{\mathrm{lim}}P_i$, where $P_i$ are finitely generated projective $R$-modules. Then, the isomorphism $SG \cong RG\otimes_{R}S \cong \underrightarrow{\mathrm{lim}}(RG\otimes_{R}P_i)$ implies that $SG$ is a flat $RG$-module, that is, $RG\rightarrow SG$ is a flat extension, as expected.
\end{proof}

\begin{lemma}\label{lem:SG-GF}
Let $R\rightarrow S$ be a flat extension of commutative rings, and $G$ be any group. For any Gorenstein flat $RG$-module $M$, the induced $SG$-module $S\otimes_{R}M$ is also Gorenstein flat.
\end{lemma}

\begin{proof}
Let $M$ and $N$ be $RG$-modules. By using the anti-automorphism $g\rightarrow g^{-1}$ of $G$, we can set $mg = g^{-1}m$ for any $g\in G$ and $m\in M$. Note that $M\otimes_{RG}N$ is obtained from $M\otimes_{R}N$ by introducing the relations $g^{-1}m\otimes n = mg\otimes n = m\otimes gn$. Then, we replace $m$ by $gm$ and obtain $m\otimes n = gm\otimes gn$. Hence, we see that $M\otimes_{RG}N = (M\otimes_{R}N)_{G}$, where $G$ acts diagonally on $M\otimes_{R}N$, and the group of co-invariants $(M\otimes_{R}N)_{G}$ is defined to be the largest quotient of $M\otimes_{R}N$ on which $G$ acts trivially.

Let $F$ be any flat left $RG$-module, which can be restricted to be a flat $R$-module. By considering $S$ as an $SG$-$R$-bimodule, we have an induced $SG$-module $S\otimes_{R} F$, where the $G$-action on $S$ is trivial. Let $M$ be any right $SG$-module, which has a natural $RG$-module structure by the ring extension $RG\rightarrow SG$. We have
$$M\otimes_{SG}(S\otimes_{R}F) = (M\otimes_{S}(S\otimes_{R}F))_{G} \cong (M\otimes_{R}F)_{G} = M\otimes_{RG}F.$$
This implies that the functor $-\otimes_{SG}(S\otimes_{R}F)$ is exact since the $RG$-module $F$ is flat and the functor  $-\otimes_{RG}F$ is exact. Hence, the induced $SG$-module $S\otimes_{R} F$ is flat.

Now assume that $M$ is a Gorenstein flat $RG$-module. Then $M$ admits a totally acyclic complex of flat $RG$-modules
$\mathbb{F} = \cdots\rightarrow F_{1}\rightarrow F_{0}\rightarrow F_{-1}\rightarrow\cdots$ such that $M \cong \mathrm{Ker}(F_0\rightarrow F_{-1})$. Since $S$ is a flat $R$-module, we obtain an acyclic complex of flat $SG$-modules
$$S\otimes_{R}\mathbb{F} = \cdots\longrightarrow S\otimes_{R}F_1\longrightarrow S\otimes_{R}F_0\longrightarrow S\otimes_{R}F_{-1}\longrightarrow \cdots$$
by applying the functor $S\otimes_{R}-$.

For any injective right $SG$-module $I$, we have an isomorphism $I\otimes_{SG}(S\otimes_{R}\mathbb{F})\cong I\otimes_{RG}\mathbb{F}$. We claim that $I$ is restricted to be an injective right $RG$-module. There are natural equivalences of functors
$$\mathrm{Hom}_{RG}(-, I) \cong \mathrm{Hom}_{RG}(-, \mathrm{Hom}_{SG}(SG, I))\cong \mathrm{Hom}_{SG}(-\otimes_{RG}SG, I),$$
where the second one is from the standard adjunction. Since $R\rightarrow S$ is assumed to be a flat extension, it follows from Lemma \ref{lem:flatext} that $SG$ is a flat $RG$-module. Moreover, by noting $I$ is an injective right $SG$-module, we infer that the functor $\mathrm{Hom}_{SG}(-\otimes_{RG}SG, I)$ is exact. Hence, $\mathrm{Hom}_{RG}(-, I)$ is an exact functor and $I$ is an injective right $RG$-module, as desired.

Then, for the totally acyclic complex of flat $RG$-modules $\mathbb{F}$, the complex $I\otimes_{RG}\mathbb{F}$ remains acyclic. Consequently, the complex $I\otimes_{SG}(S\otimes_{R}\mathbb{F})$ is acyclic, i.e. $S\otimes_{R}\mathbb{F}$ is a totally acyclic complex of flat $SG$-modules, and moreover, $S \otimes_{R}M \cong \mathrm{Ker}(S\otimes_{R}F_0\rightarrow S\otimes_{R}F_{-1})$ is a Gorenstein flat $SG$-module. This completes the proof.
\end{proof}

Now, we are in a position to state the following.

\begin{theorem}\label{thm:ringext}
Let $R\rightarrow S$ be a flat extension of commutative rings. For any group $G$, we have $${\rm Ghd}_{S}G\leq {\rm Ghd}_{R}G.$$
\end{theorem}

\begin{proof}
There is nothing to prove if ${\rm Ghd}_{R}G = \infty$. Then, it only suffices to consider the case where ${\rm Ghd}_{R}G = n$ is finite. By \cite[Theorem 3.17]{Hol04}, there exists an exact sequence of $RG$-modules $0\rightarrow K\rightarrow M\rightarrow R\rightarrow 0$, where $M$ is Gorenstein flat, ${\rm fd}_{RG}K = n-1$. By applying the functor $S\otimes_{R}-$, we obtain an exact sequence of induced $SG$-modules $$0\longrightarrow S\otimes_{R}K \longrightarrow S\otimes_{R}M \longrightarrow S\longrightarrow 0.$$
By Lemma \ref{lem:SG-GF}, $S\otimes_{R}M$ is a Gorenstein flat $SG$-module. Remark that for any flat $RG$-module $F$, the induced $SG$-module $S\otimes_{R} F$ is also flat. Then, ${\rm fd}_{SG}(S\otimes_{R}K) = n-1$, and we infer from the above exact sequence that ${\rm Ghd}_{S}G = {\rm Gfd}_{SG}S \leq n$. This completes the proof.
\end{proof}

The following is immediate, which implies that under a quite mild condition, the Gorenstein homological dimension ${\rm Ghd}_{R}G$ of a group $G$ over any commutative ring $R$ is a refinement of ${\rm Ghd}_{\mathbb{Z}}G$, the one introduced in \cite{ABHS11} over the ring of integers $\mathbb{Z}$.

\begin{corollary}\label{cor:GhdZG}
Let $R$ be a commutative ring, $G$ a group. If $R$ is a $\mathbb{Z}$-flat module ($\mathbb{Z}$-torsion-free), then ${\rm Ghd}_{R}G\leq {\rm Ghd}_{\mathbb{Z}}G$. In particular, ${\rm Ghd}_{\mathbb{Q}}G\leq {\rm Ghd}_{\mathbb{Z}}G$.
\end{corollary}

\begin{proposition}\label{prop:equ}
Let $R$ be a commutative ring, $G$ be a countable group. Let $S = R[x]/(x^{n})$ be the quotient of the polynomial ring, where $n > 1$ is an integer, and $x$ is a variable which is supposed to commute with all the elements of $G$. Then we have $\mathrm{Ghd}_{R}G = \mathrm{Ghd}_{S}G$.
\end{proposition}

\begin{proof}
Note that $R\rightarrow S$ is a flat extension, and then $\mathrm{Ghd}_{S}G \leq \mathrm{Ghd}_{R}G$ follows by Theorem \ref{thm:ringext}. We consider $R$ and $S$ as $RG$-modules with trivial $G$-actions, then $R$ is a direct summand of $S = R[x]/(x^{n})$, and $\mathrm{Ghd}_{R}G = \mathrm{Gfd}_{RG}R \leq \mathrm{Gfd}_{RG}S$. It remains to prove the inequality $\mathrm{Gfd}_{RG}S \leq \mathrm{Ghd}_{S}G = \mathrm{Gfd}_{SG}S$.

Observe that $SG = RG[x]/(x^n)$, which is easily seen from the equation
$$\sum_{j=0}^{n-1}(\sum_{i\in \mathbb{N}}r_{ij}g_i)x^j =  \sum_{i\in \mathbb{N}}(\sum_{j=0}^{n-1}r_{ij}x^j)g_i,~~~~~~~~ \forall r_{ij}\in R, g_i\in G.$$
Let $M$ be a Gorenstein flat $SG$-module and $\mathbb{F}= \cdots\rightarrow F_{1}\rightarrow F_{0}\rightarrow F_{-1}\rightarrow\cdots$ be a totally acyclic complex of flat $SG$-modules such that $M \cong \mathrm{Ker}(F_0\rightarrow F_{-1})$. Let $I$ be any injective right $RG$-module. There are isomorphisms $I\otimes_{RG}SG \cong {\rm Hom}_{RG}(SG, I)$ and $I\otimes_{RG}\mathbb{F}\cong I\otimes_{RG}SG\otimes_{SG}\mathbb{F}$. We imply that the induced right $SG$-module  $I\otimes_{RG}SG$ is injective, and the complex $I\otimes_{RG}\mathbb{F}$ is acyclic. That is, by restriction we can obtain a totally acyclic complex $\mathbb{F}$ of flat $RG$-modules, and hence $M$ is restricted to be a Gorenstein flat $RG$-module.

The inequality $\mathrm{Gfd}_{RG}S \leq \mathrm{Ghd}_{S}G$ is obviously true if $\mathrm{Ghd}_{S}G = \infty$. Now we assume that $\mathrm{Ghd}_{S}G = n$ is finite. There is an exact sequence of $SG$-modules
$$0\longrightarrow M_n\longrightarrow M_{n-1}\longrightarrow \cdots\longrightarrow M_1\longrightarrow M_0\longrightarrow S\longrightarrow 0,$$
where all $M_i$ are Gorenstein flat. By the above argument, we may consider $M_i$ as Gorenstein flat $RG$-modules, and hence, $\mathrm{Gfd}_{RG}S \leq  n$. This completes the proof.
\end{proof}

\section {\bf A version of Serre's theorem}

In this section, we consider the behavior of Gorenstein homological dimension of subgroups. Let $H$ be a subgroup of $G$ with finite index. Recall that there is an inequality ${\rm Ghd}_{\mathbb{Z}}H\leq {\rm Ghd}_{\mathbb{Z}}G$, and moreover, the equality ${\rm Ghd}_{\mathbb{Z}}H = {\rm Ghd}_{\mathbb{Z}}G$ holds provided the supremum of injective length (dimension) of flat $\mathbb{Z}G$-modules, denoted by ${\rm silf}(\mathbb{Z}G)$, is finite; see \cite[Proposition 4.11]{ABHS11} and \cite[Theorem 4.18]{ABHS11} respectively.

We have the following result, where the finiteness of ``${\rm silf}$'' is not necessarily needed; see Theorem \ref{thm:groupext}. Hence, it strengthens and extends \cite[Proposition 4.11]{ABHS11} and \cite[Theorem 4.18]{ABHS11}. By Serre's Theorem, there is an equality between cohomology dimensions of any torsion-free group and its subgroup of finite index; see details in \cite[Theorem VIII 3.1]{Bro82}. In this sense, the result can also be regarded as a Gorenstein homological version of Serre's Theorem.

Let $H$ be any subgroup of $G$. There exist simultaneously an induction functor $\mathrm{Ind}_H^G = RG\otimes_{RH}-$ and a coinduction functor $\mathrm{Coind}_H^G = \mathrm{Hom}_{RH}(RG, -)$ from the category of $RH$-modules ${\rm Mod}(RH)$ to the category of $RG$-modules ${\rm Mod}(RG)$. We denote by ${\rm Res}_H^G$ the standard restriction functor, which sends every $RG$-module to be an $RH$-module.

We have the following observations.

\begin{lemma}\label{lem:2facts}
Let $G$ be a group, and $H$ be any subgroup of $G$ with finite index. For any $RG$-module $M$, the following hold.
\begin{enumerate}
\item If $M$ is a Gorenstein flat module, then the restricted $RH$-module ${\rm Res}_H^GM$ is also Gorenstein flat.
\item The $RH$-module ${\rm Res}_H^GM$ is Gorenstein flat, if and only if the induced $RG$-module ${\rm Ind}_H^G{\rm Res}_H^GM$ is also Gorenstein flat.
\end{enumerate}
\end{lemma}

\begin{proof}
(1) For the Gorenstein flat $RG$-module $M$, let $\mathbb{F}= \cdots\rightarrow F_{1}\rightarrow F_{0}\rightarrow F_{-1}\rightarrow\cdots$ be a totally acyclic complex of flat $RG$-modules such that $M \cong \mathrm{Ker}(F_0\rightarrow F_{-1})$. By restriction, we obtain an acyclic complex ${\rm Res}_H^G\mathbb{F}$ of flat $RH$-modules. Since the index $[G:H]$ is finite, there is an equivalence of functors $\mathrm{Ind}_H^G \simeq \mathrm{Coind}_H^G$; see for example \cite[Proposition III 5.9]{Bro82}. For any injective right $RH$-module $I$, the $RG$-module $\mathrm{Ind}_H^GI \cong \mathrm{Coind}_H^GI$ is injective, and then the complex ${\rm Ind}_H^GI\otimes_{RG}\mathbb{F}$ is acyclic. We infer from the isomorphism $I\otimes_{RH}{\rm Res}_H^G\mathbb{F}\cong {\rm Ind}_H^GI\otimes_{RG}\mathbb{F}$ that the complex $I\otimes_{RH}{\rm Res}_H^G\mathbf{F}$ is acyclic, and then ${\rm Res}_H^G\mathbb{F}$ is a totally acyclic complex of flat $RH$-modules. Consequently, $M$ is restricted to be a Gorenstein flat $RH$-module, as claimed.

(2) For the ``only if'' part, note that $RG$ is a projective $RH$-module. For the Gorenstein flat $RH$-module ${\rm Res}_H^GM$, there is a totally acyclic complex of flat $RH$-modules $\mathbb{F}$. Then, we obtain an induced complex ${\rm Ind}_H^G\mathbb{F}$, which is acyclic with each item being flat $RG$-module. Let $E$ be any injective right $RG$-module. The restricted $RH$-module ${\rm Res}_H^GE$ is injective, and we infer from the isomorphism $E\otimes_{RG}{\rm Ind}_H^G\mathbb{F} \cong {\rm Res}_H^GE\otimes_{RH}\mathbb{F}$ that the complex
${\rm Ind}_H^G\mathbb{F}$ is totally acyclic. Hence, the induced module ${\rm Ind}_H^G{\rm Res}_H^GM$ is a Gorenstein flat $RG$-module. Here, we do not need to assume that the index $[G:H]$ is finite.

For any $RG$-module $M$, there is a canonical $RG$-map $${\rm Ind}_H^G{\rm Res}_H^GM = RG\otimes_{RH}{\rm Res}_H^GM\longrightarrow M$$ given by $g\otimes m\mapsto gm$. This map is surjective; moreover, as an $RH$-map it is a split surjective. If the $RG$-module ${\rm Ind}_H^G{\rm Res}_H^GM$ is Gorenstein flat, we infer from (1) that ${\rm Res}_H^G{\rm Ind}_H^G{\rm Res}_H^GM$ is a Gorenstein flat $RH$-module. Consequently, by Lemma \ref{lem:gf-closed} its direct summand ${\rm Res}_H^GM$ is a Gorenstein flat $RH$-module. This proves the ``if'' part.
\end{proof}

\begin{theorem}\label{thm:groupext}
Let $G$ be a group, $H$ be a subgroup of $G$ with finite index. There is an equality
$${\rm Ghd}_{R}G = {\rm Ghd}_{R}H.$$
\end{theorem}

\begin{proof}
To simplify the burden of notations, for any $RG$-module $M$, the restricted $RH$-module ${\rm Res}_H^GM$ will be denoted by $M$ as well. It follows immediately from the first assertion of Lemma \ref{lem:2facts} that for any $RG$-module $N$, one has an inequality ${\rm Gfd}_{RH}N\leq {\rm Gfd}_{RG}N$. In particular, for the trivial $RG$-module $R$, we have $\mathrm{Ghd}_{R}H \leq \mathrm{Ghd}_{R}G$.

It remains to prove ${\rm Ghd}_{R}G\leq {\rm Ghd}_{R}H$. Since the inequality is trivial if ${\rm Ghd}_{R}H = \infty$, it suffices to  assume ${\rm Ghd}_{R}H = n$ is finite. Take an exact sequence $0\rightarrow N\rightarrow F_{n-1}\rightarrow \cdots\rightarrow F_0\rightarrow R\rightarrow 0$ of $RG$-modules with each $F_i$ being flat. Since $F_i$ are restricted to be flat $RH$-modules, it follows from ${\rm Ghd}_{R}H = n$ that as a restricted $RH$-module, $N$ is Gorenstein flat. For the required inequality, it suffices to show that $N$ is a Gorenstein flat $RG$-module.

Let $I$ be any injective right $RG$-module. There is a canonical map of $RG$-modules
$$I\rightarrow {\rm Hom}_{RH}(RG, I) = {\rm Coind}_H^GI$$
given by $x\mapsto (g\mapsto gx)$ for any $x\in I$ and any $g\in G$. This map is injective, and then $I$ is a direct summand of ${\rm Coind}_H^GI$. For any $i>0$, we have ${\rm Tor}_i^{RH}(I, N) = 0$ since by restriction $N$ is a Gorenstein flat $RH$-module, and $I$ is an injective $RH$-module. Then, we infer from the isomorphism
$${\rm Tor}_i^{RG}({\rm Coind}_H^GI, N)\cong {\rm Tor}_i^{RG}({\rm Ind}_H^GI, N)\cong {\rm Tor}_i^{RH}(I, N)$$
that ${\rm Tor}_i^{RG}(I, N) = 0$.

Let $\alpha: N\rightarrow {\rm Ind}_H^GN$ be a composition of the canonical $RG$-map $N\rightarrow {\rm Coind}_H^GN$, followed by the isomorphism ${\rm Coind}_H^GN\rightarrow {\rm Ind}_H^GN$. Then $\alpha$ is an $RG$-monic and is split as an $RH$-map. Since $N$ is a Gorenstein flat $RH$-module, it follows from Lemma \ref{lem:2facts} that the induced module ${\rm Ind}_H^GN$ is a Gorenstein flat $RG$-module. Hence, there is an exact sequence of $RG$-modules
$0\rightarrow {\rm Ind}_H^GN\stackrel{\beta}\rightarrow F_0\rightarrow {\rm Coker}\beta\rightarrow 0$
for which $F_0$ is flat and ${\rm Coker}\beta$ is Gorenstein flat.

Consider the following diagram
$$\begin{xymatrix}@C=20pt{
0 \ar[r] &N \ar[r]^{\gamma} \ar[d]_{\alpha} &F_0 \ar[r] \ar@{=}[d] & {\rm Coker}\gamma \ar[r]\ar[d] & 0\\
0 \ar[r] &{\rm Ind}_H^GN \ar[r]^{\quad\beta}\ar[r] &F_0 \ar[r] &{\rm Coker}\beta \ar[r] & 0
}\end{xymatrix}$$
Since $\alpha: N\rightarrow {\rm Ind}_H^GN$ is a split $RH$-monic, the map $I\otimes_{RH}\alpha$ is injective. Since
${\rm Coker}\beta$ is restricted to be a Gorenstein flat $RH$-module and $I$ is restricted to be an injective $RH$-module, we have
${\rm Tor}_1^{RH}(I, {\rm Coker}\beta) = 0$, which implies that the map $I\otimes_{RH}\beta$ is injective. Hence, the map $I\otimes_{RH}\gamma$ is injective, as well. We infer from the exact sequence
$$0\longrightarrow {\rm Tor}_1^{RH}(I, {\rm Coker}\gamma)\longrightarrow I\otimes_{RH} N\longrightarrow I\otimes_{RH} F_0\longrightarrow I\otimes_{RH} {\rm Coker}\gamma\longrightarrow 0$$
that ${\rm Tor}_1^{RH}(I, {\rm Coker}\gamma) = 0$, and moreover, it yields by \cite[Proposition 3.8]{Hol04} and \cite[Corollary 4.12]{SS20} that the restricted $RH$-module ${\rm Coker}\gamma$ is Gorenstein flat. Analogous to the above argument, we have
${\rm Tor}_1^{RG}({\rm Coind}_H^GI, {\rm Coker}\gamma) \cong {\rm Tor}_1^{RH}(I, {\rm Coker}\gamma)$, and $I$ is a direct summand of ${\rm Coind}_H^GI$ as $RG$-modules. Hence ${\rm Tor}_1^{RG}(I, {\rm Coker}\gamma) = 0$, and furthermore, the sequence $$0\longrightarrow N\longrightarrow F_0\longrightarrow{\rm Coker}\gamma\longrightarrow 0$$ remains exact after applying $I\otimes_{RG}-$.

Then, repeat the above argument for ${\rm Coker}\gamma$, we will obtain inductively an acyclic complex
$$0\longrightarrow N\longrightarrow F_0\longrightarrow F_{-1}\longrightarrow F_{-2}\longrightarrow \cdots$$
with each $F_i$ a flat $RG$-module, which remains acyclic after applying $I\otimes_{RG}-$ for any injective right $RG$-module $I$. By pasting this sequence with the flat resolution of $N$, we will obtain a totally acyclic complex of flat $RG$-modules for $N$, and then $N$ is a Gorenstein flat $RG$-module, as expected. This completes the proof.
\end{proof}

We have the following immediately, which extends \cite[Proposition 4.12]{ABHS11}.

\begin{corollary}\label{cor:fgroup}
Let $G$ be a group. The following conditions are equivalent$:$
\begin{enumerate}
\item $G$ is a finite group.
\item For any commutative ring $R$, ${\rm Ghd}_{R}G = 0$.
\item ${\rm Ghd}_{\mathbb{Z}}G = 0$.
\end{enumerate}
\end{corollary}

\begin{proof}
(1)$\Longrightarrow$(2). Let $H = \{e\}$ be the subgroup of $G$, which only contains the identity element of $G$. Since $G$ is a finite group, the index $[G: H]$ is finite. Hence, ${\rm Ghd}_{R}G = {\rm Ghd}_{R}H = 0$ follows immediately from Theorem \ref{thm:groupext}.

(2)$\Longrightarrow$(3) is trivial, and (3)$\Longleftrightarrow$(1) is precisely \cite[Proposition 4.12]{ABHS11}.
\end{proof}

\begin{remark}\label{rem:FroExt}
If $H$ is a subgroup of $G$ with finite index, then both $(\mathrm{Ind}_H^G, {\rm Res}_H^G)$ and $({\rm Res}_H^G, \mathrm{Ind}_H^G)$ are adjoint pairs of functors. For any associative ring $R$, $RH\rightarrow RG$ is a Frobenius extension of group rings, and the pair of functors $(\mathrm{Ind}_H^G, {\rm Res}_H^G)$ is called a strongly adjoint pair by Morita \cite{Mor65}, or a Frobenius pair by \cite[Definition 1.1]{CIGTN99}. The Gorenstein homological properties of modules under Frobenius extension of rings and Frobenius pairs of functors were studied in \cite{CR22, HLGZ20, Ren18, Ren18+1, Ren19, Zhao}.
\end{remark}

\begin{fund*}
The second author is supported by the National Natural Science Foundation of China (No. 11871125) and Natural Science Foundation of Chongqing, China (No. cstc2018jcyjAX0541).
\end{fund*}

\end{document}